\pdfoutput=1 

\RequirePackage[l2tabu, orthodox]{nag}

\documentclass[12pt,reqno]{amsart}
\usepackage[margin=1in]{geometry} 
\usepackage{url,amssymb,enumerate,colonequals,graphicx,booktabs,bbm}
\usepackage{mathrsfs} 

\usepackage[T1]{fontenc}
\usepackage[utf8]{inputenc}
\usepackage[dvipsnames,xcdraw]{xcolor}

\newcommand{\deficolor}{ForestGreen}
\newcommand{\defi}[1]{\textbf{\color{\deficolor} #1}} 

\newcommand{\Aff}{\mathbb{A}}
\newcommand{\C}{\mathbb{C}}
\newcommand{\F}{\mathbb{F}}

\newcommand{\PP}{\mathbb{P}}
\newcommand{\Q}{\mathbb{Q}}
\newcommand{\R}{\mathbb{R}}
\newcommand{\Z}{\mathbb{Z}}
\newcommand{\Qbar}{{\overline{\Q}}}
\newcommand{\Zhat}{{\hat{\Z}}}

\newcommand{\kbar}{{\overline{k}}}

\newcommand{\calC}{\mathcal{C}}

\newcommand{\calF}{\mathcal{F}}

\newcommand{\calH}{\mathcal{H}}

\newcommand{\calO}{\mathcal{O}}
\newcommand{\calP}{\mathcal{P}}
\newcommand{\calQ}{\mathcal{Q}}

\newcommand{\calX}{\mathcal{X}}

\newcommand{\cc}{\mathfrak{c}}
\newcommand{\ff}{\mathfrak{f}}
\newcommand{\GG}{\mathfrak{G}}
\newcommand{\pp}{\mathfrak{p}}

\newcommand{\qq}{\mathsf{q}}

\DeclareMathOperator{\AGL}{AGL}
\DeclareMathOperator{\Aut}{Aut}
\DeclareMathOperator{\disc}{disc}
\DeclareMathOperator{\divv}{div}

\DeclareMathOperator{\Frob}{Frob}
\DeclareMathOperator{\Gal}{Gal}

\DeclareMathOperator{\HH}{H}
\DeclareMathOperator{\Hom}{Hom}

\DeclareMathOperator{\Inn}{Inn}
\DeclareMathOperator{\Ni}{Ni}
\DeclareMathOperator{\sma}{small}
\DeclareMathOperator{\Spec}{Spec}
\DeclareMathOperator{\Stab}{Stab}
\DeclareMathOperator{\Tr}{Tr}
\DeclareMathOperator{\wt}{wt}

\newcommand{\injects}{\hookrightarrow}
\newcommand{\intersect}{\cap}
\newcommand{\isom}{\simeq}
\newcommand{\isomto}{\stackrel{\sim}\to}
\newcommand{\surjects}{\twoheadrightarrow}
\newcommand{\tensor}{\otimes}

\newcommand{\arith}{\mathrm{arith}}
\newcommand{\et}{\mathrm{\acute{e}t}}
\newcommand{\geom}{\mathrm{geom}}

\newtheorem{theorem}{Theorem}[section]
\newtheorem{lemma}[theorem]{Lemma}
\newtheorem{corollary}[theorem]{Corollary}
\newtheorem{proposition}[theorem]{Proposition}

\theoremstyle{definition}

\newtheorem{conjecture}[theorem]{Conjecture}
\newtheorem{example}[theorem]{Example}

\theoremstyle{remark}
\newtheorem{remark}[theorem]{Remark}

\makeatletter
\g@addto@macro\bfseries{\boldmath} 
\makeatother

\usepackage{microtype}  

\usepackage[
	pagebackref,
	pdfauthor={Xiaoyu Huang, Blake Jackson, Kyu-Hwan Lee, Bjorn Poonen, Rachel Pries, Shaowu Zhang}, 
	pdftitle={The Mathieu group M23 is a Galois group over Q},
]{hyperref}

\begin{document}

\title{The Mathieu group $M_{23}$ is a Galois group over $\Q$}



\author[X. Huang]{Xiaoyu Huang}
\address{Department of Mathematics, Temple University, Philadelphia, PA 19122, USA} \email{xiaoyu.huang@temple.edu}
\urladdr{\url{https://sites.google.com/view/xiaoyuhuang/home}}

\author[B. Jackson]{Blake Jackson}
\address{Institute for Computer-Aided Reasoning in Mathematics, Carnegie Mellon University, Pittsburgh, PA 15213, USA}
\email{jackson@icarm.io}
\urladdr{\url{https://www.blakejacksonmath.com/}}

\author[K.-H. Lee]{Kyu-Hwan Lee}
\address{Department of Mathematics, University of Connecticut, Storrs, CT 06269, USA \hfill \break \indent Korea Institute for Advanced Study, Seoul 02455, Republic of Korea}
\email{khlee@math.uconn.edu}
\urladdr{\url{https://khlee-math.github.io/}}

\author[B. Poonen]{Bjorn Poonen}
\address{Department of Mathematics, Massachusetts Institute of Technology, Cambridge, MA 02139-4307, USA}
\email{poonen@math.mit.edu}
\urladdr{\url{http://math.mit.edu/~poonen/}}

\author[R. Pries]{Rachel Pries}
\address{Department of Mathematics, Colorado State University, Fort Collins, CO 80523, USA}
\email{pries@colostate.edu}
\urladdr{\url{http://math.colostate.edu/~pries/}}

\author[S. Zhang]{Shaowu Zhang}
\address{Department of Mathematics, California Institute of Technology, Pasadena, CA 91125, USA}
\email{szhang7@caltech.edu}
\urladdr{\url{https://www.shaowuzhang.com/}}

\thanks{
X.H.\ was supported in part by an AMS--Simons Travel Grant and NSF grant DMS-2617281. B.J.\ was supported by ICARM through NSF grant DMS-2425401. K.-H.L.\ was supported by NSF grant DMS-2617278. B.P.\ was supported in part by NSF grant DMS-2601551 and Simons Foundation grant \#402472.
R.P.\ was supported by NSF grant DMS-2200418.
Several authors received travel support from AIM}

\date{October 4, 2026}

\begin{abstract}
Researchers studying the inverse Galois problem realized 25 of the 26 sporadic finite simple groups 
as Galois groups over $\Q$ during 1984--1989.
We complete this program by proving that the last remaining sporadic group, the Mathieu group $M_{23}$, 
occurs as a Galois group over $\Q$.
In fact, we produce an explicit degree $23$ polynomial with rational 
coefficients whose splitting field has Galois group $M_{23}$ over $\Q$. 
To accomplish this, we discover an unanticipated splitting of the  
Nielsen class associated to a non-rigid triple of conjugacy classes of $M_{23}$.  
We compute the Belyi maps associated with this Nielsen class to construct an explicit regular Galois extension of $\Q(t)$ with Galois group $M_{23}$.
Essential for our computation is the numerical Belyi map algorithm developed and implemented by Costa, Klug, Musty, Schiavone, Sijsling, and Voight, inspired by ideas of Hejhal and Stark.
As a bonus, one specialization of the regular extension yields an $M_{22}$-extension of $\Q$.
\end{abstract}

\maketitle

\section{Introduction}\label{S:introduction}

\subsection{The inverse Galois problem}
\label{S:inverse Galois problem}

The inverse Galois problem asks whether or not every finite group $G$ occurs as the Galois group of some finite Galois extension of $\Q$ \cite{Serre1992,Malle-Matzat2018,Wittenberg2025}.
In particular, one can ask the question for finite simple groups, which fall into $18$ infinite families and $26$ ``sporadic'' groups.
During 1984--1989, $25$ of the $26$ sporadic groups were realized as Galois groups over $\Q$.\footnote{Strictly speaking, the baby monster $B$ was realized only later, since there was a calculation error in the original construction of \cite{Hunt1986}, according to \cite[p.~161]{Malle-Matzat2018}.}
The last sporadic group, the Mathieu group $M_{23}$, remained unrealized, despite many attempts; see Section~\ref{S:previous}.
In this paper, we realize this group. 

\begin{theorem}
\label{T:over Q}
There exists a Galois extension of $\Q$ with Galois group $M_{23}$.
\end{theorem}

\begin{example}
\label{Ex:M23 over Q}
The splitting field of 
\begin{align*}
    x^{23} &- 184 x^{21} - 1150 x^{20} + 26151 x^{19} + 18400 x^{18} - 1808490 x^{17} \\
      &+ 1545462 x^{16} + 67672923 x^{15} - 42732528 x^{14} - 1333395744 x^{13} \\
      &+ 290615166 x^{12} + 10550424369 x^{11} + 3700476348 x^{10} + 35123826654 x^9 \\
      &- 194398310718 x^8 - 1023887308293 x^7 + 3961650395556 x^6 \\
      &+ 1949980486716 x^5 - 28142323927002 x^4 + 53599151839311 x^3 \\
      &- 46185312415788 x^2 + 19169943578802 x - 3150159884154
\end{align*}
is a Galois extension of $\Q$ with Galois group $M_{23}$, unramified outside $\{2,3,23\}$. We certified that the Galois group is $M_{23}$ with the \texttt{GaloisProof} command in Magma~\cite{Magma}.\footnote{Our code is available at \url{https://github.com/shaowuz/m23isgalois}.}
\end{example}



Prior to our work, it was known that 4952 of the 4953 transitive permutation groups of degree $\le 23$ were realizable over $\Q$ \cite{Bommel-Costa-Elkies-Keller-Schiavone-Voight-arXiv}.
Theorem~\ref{T:over Q} realizes the last remaining group in this program too.
According to \cite{sair_igp24}, all 25000 transitive permutation groups of degree $24$ have now been realized too.

There remain infinitely many finite simple groups of Lie type that have not been realized over $\Q$, especially those associated to non-prime finite fields.
For a list of the unrealized ones of size $\le 10^8$, see \cite{Zywina-small-groups}.
The inverse Galois problem for arbitrary finite groups remains far from being resolved.

\subsection{Regular Galois extensions}

A \defi{$G$-extension} of a field $k$ is a Galois field extension $L/k$, together with an isomorphism $\Gal(L/k) \isom G$.
An extension of $k$ is \defi{regular} if it contains no nontrivial algebraic extension of $k$ (and is linearly disjoint from $k^{1/p}$ over $k$ if $k$ is of characteristic $p$).
By a slight abuse of terminology, we call a finite extension of $k(t)$ \defi{regular} if it is regular over $k$.
If $G \ne \{1\}$ and a regular $G$-extension of $\Q(t)$ exists,
then, by Hilbert's irreducibility theorem, 
specializing $t$ to different rational numbers yields infinitely many $G$-extensions of $\Q$.
Thus, to prove Theorem~\ref{T:over Q}, we prove the following stronger result. 

\begin{theorem}
\label{T:over Q(t)}
There exists a regular Galois extension of $\Q(t)$ with Galois group $M_{23}$.
\end{theorem}

\subsection{Previous work}
\label{S:previous}

The rigidity method, originally developed in \cite{Shih1974,Fried1977,Belyi1979,Matzat1979,Matzat1984,Matzat1985a,Thompson1984},
was instrumental for constructing regular Galois extensions of $\Q(t)$.
Variants of this method, often involving significant new ideas such as the braid group method of \cite{Matzat1989,Matzat1991,Fried-Voelklein1991}, were used to construct regular $G$-extensions of $\Q(t)$ for all the sporadic groups $G$ except $M_{23}$ \cite{Thompson1984,Matzat1985a,Hoyden-Siedersleben1985, Hunt1986, Hoyden-Siedersleben-Matzat1986, Matzat-Zeh-Marschke1986, Matzat-Zeh-Marschke1987, Malle88,
Pahlings1988,Pahlings1989}; see \cite[II.9]{Malle-Matzat2018} for a detailed exposition.

Prior attempts to realize $M_{23}$ over $\Q$ resulted only in realizations over other fields:
\begin{itemize}
\item Hoyden-Siedersleben~\cite{Hoyden-Siedersleben1985} and H\"afner~\cite{Haefner1987} used regular $M_{24}$-extensions of $k(t)$ to construct regular $M_{23}$-extensions of $k(t)$ for $k=\Q(\sqrt{-23})$ and $k=\Q(\sqrt{-7})$, respectively.
\item Granboulan \cite{Granboulan1996} used another $M_{24}$-extension to construct a regular $M_{23}$-extension over $k(t)$ for any field $k$ over which a certain conic $C$ has a $k$-point.
Unfortunately, $C$ has no $\Q$-point.
\item Elkies \cite{Elkies2013} calculated the four degree~23 polynomials $P(x) \in \C[x]$ (up to affine equivalence) such that the Galois group of $P(x)-t$ over $\C(t)$ is $M_{23}$.
Each such $P(x)$ has coefficients in a degree~$4$ number field $F$, 
so one obtains a regular $M_{23}$-extension of $F(t)$.
\end{itemize}
For more information and explicit computation of Hurwitz spaces for $M_{23}$-covers, see \cite{Koenig2014}, \cite{Haefner-preprint}, and \cite{Seguin-preprint}.
Finally, the problem of realizing $M_{23}$ over $\Q$ was on the Epoch AI FrontierMath open problem list \cite{EpochAI} as a benchmark for AI capability since at least January 2026.

\subsection{Outline of our approach}

\begin{enumerate}[\upshape 1.]
\item The group $M_{23}$ has $17$ conjugacy classes, traditionally labeled by the order of a representative element followed by an uppercase letter to distinguish classes of the same order.
The group $\GG_\Q \colonequals \Gal(\Qbar/\Q)$ acts on the set of conjugacy classes of $M_{23}$.
It fixes the class 2 $\colonequals$ 2A, and acts on 23A and 23B nontrivially through the quotient $\Gal(K/\Q)$, where $K \colonequals \Q(\sqrt{-23})$.
We explain in Section~\ref{S:M23} why we looked at \emph{these} classes.

\item
Let $S \subset \PP^1_\Q$ be the finite \'etale $\Q$-scheme whose geometric points are $0,\sqrt{-23},-\sqrt{-23}$.

\item 
The Riemann existence theorem, together with a calculation in $M_{23}$, shows that there are exactly seven $M_{23}$-covers $Y_\C \to \PP^1_\C$ ramified only over $S$ with local monodromy (2,23A,23B).
The quotient of each $Y_\C$ by the one-point stabilizer $M_{22}$ is a genus~$4$ curve $X_\C$,
and $Y_\C \to \PP^1_\C$ is the Galois closure of the degree~$23$ cover $X_\C \to \PP^1_\C$.

\item
We compute 
these seven curves $X_\C$ numerically using the BelyiDB package \cite{BelyiDB},
and recognize coefficients as algebraic numbers 
to compute \emph{candidates} $X_1,\ldots,X_7$ for these curves over $\Qbar$.

\item
To verify the $X_i$, we construct a degree $23$ map $X_i \to \PP^1$ for each $i$, verify that it is ramified only over $S$ with local monodromy (2,23A,23B),
verify that its geometric monodromy group is $M_{23}$ by working modulo a prime of good reduction, and verify that $X_1,\ldots,X_7$ are pairwise non-isomorphic over $\Qbar$.

\item 
One might expect $\GG_\Q$ to act transitively on $\{X_1,\ldots,X_7\}$.
Miraculously, instead the $\GG_\Q$-orbits have sizes $6$ and $1$! 
We do not have a conceptual explanation for this. 

\item
We prove that the $\GG_\Q$-fixed point descends to a cover $X \stackrel{t}\to \PP^1_\Q$ over $\Q$ by computing an explicit birational map from $X$ to a (singular) plane curve defined over $\Q$ under which $t$ corresponds to the first coordinate projection.

\item
Taking the Galois closure of the corresponding degree~$23$ field extension
gives a regular $M_{23}$-extension of $\Q(t)$.
\end{enumerate}

Details will be given in the subsequent sections.

\begin{remark}
That the $M_{23}$-cover corresponding to the $\GG_\Q$-fixed point descends to $\Q$
also follows from \cite[Proposition~2.8(c)]{Coombes-Harbater1985}, since the center $Z(M_{23})$ is trivial.
\end{remark}

\subsection{Consequences}

Let $k$ be a number field.
Call a collection of $G$-extensions of $k$ \defi{mutually independent} if, for any $r \ge 1$, any $r$ distinct extensions from the collection generate a $G^r$-extension of $k$.

\begin{corollary}
Let $G$ be a sporadic finite simple group.
Let $k$ be any number field.
\begin{enumerate}[\upshape (a)]
\item \label{I:sporadic over Q(t)}
    There exists a regular $G$-extension of $k(t)$.
\item \label{I:sporadic over Q}
    There exists a $G$-extension of $k$.
\item \label{I:linearly disjoint}
    There exists a $G$-extension of $\Q$ linearly disjoint from $k$.
\item \label{I:mutually independent}
    There exists an infinite collection of mutually independent $G$-extensions of $k$.
\end{enumerate}
\end{corollary}

\begin{proof}
Part~\eqref{I:sporadic over Q(t)} over $\Q$ is the combination of Theorem~\ref{T:over Q(t)} and the work cited in Section~\ref{S:previous}.
Take a compositum to get \eqref{I:sporadic over Q(t)} over $k$.
For (\ref{I:sporadic over Q},\ref{I:linearly disjoint},\ref{I:mutually independent}) and other consequences of \eqref{I:sporadic over Q(t)}, see \cite[\S3.3--3.4]{Serre1992} and \cite[\S10.1]{Serre1997}.
\end{proof}

\section{The rigidity method}

\subsection{\texorpdfstring{$G$}{G}-covers of the projective line}

Let $G$ be a nonabelian finite simple group, so $Z(G)=1$.
Let $\Inn(G)$ be the group of inner automorphisms of $G$, so $\Inn(G) \isom G$, and $\Inn(G)$ acts on $G$, on the right, say: $g$ acts as $x \mapsto g^{-1} x g$.
The set $\calC \colonequals G/\Inn(G)$ is the set of conjugacy classes of $G$.
Let $n \ge 1$ be such that $g^n=1$ for all $g \in G$.

Let $k$ be a field of characteristic~$0$.
For each $m \ge 1$, let $\mu_m = \{x \in \kbar^\times : x^m=1\}$.
Let $\GG_k = \Gal(\kbar/k)$.
Let $\epsilon \colon \GG_k \to \Aut (\mu_n) \isom (\Z/n\Z)^\times$ be the cyclotomic character describing the $\GG_k$-action on $n$th roots of unity.
A choice of generator $\zeta \in \mu_n$ determines a bijection $\Hom(\mu_n,G) \isom \Hom(\Z/n\Z,G) = G$.
To make this bijection \emph{$\GG_k$-equivariant}, we let $\sigma \in \GG_k$ act trivially on the target $G$ in $\Hom(\mu_n,G)$ and as $g \mapsto g^{\epsilon(\sigma)^{-1}}$ on the set $G$ on the right.
Taking $\Inn(G)$-orbits, the latter action induces a $\GG_k$-action on $\calC$.

A \defi{curve} over $k$ is a $1$-dimensional smooth projective geometrically connected scheme of finite type over $k$.
A \defi{$G$-cover of $\PP^1_k$} is a (possibly ramified) Galois morphism of curves $Y \to \PP^1_k$ together with an isomorphism $\Gal(k(Y)/k(\PP^1)) \isom G$.
Our convention is that groups act on the left on function field elements and on the right on geometric objects,
so $G$ acts on the left on $k(Y)$, and on the right on $Y$.
An \defi{isomorphism of $G$-covers} $Y \to \PP^1_k$ and 
$Y' \to \PP^1_k$ is a $G$-equivariant isomorphism $Y \to Y'$ over $\PP^1_k$.

\subsection{Local monodromy}

Let $D \subset \C$ be the open unit disk.
Let $D^\times = D -\{0\}$.
We have $\Z \isomto \pi_1(D^\times)$ sending $1$ to the standard counterclockwise generator.

Suppose that $Y_\C \to \PP^1_\C$ is a $G$-cover with branch locus $S$.
For each $s \in S$, we have an analytic open embedding $D \to \PP^1_\C$ sending $0$ to $s$, 
such that $s$ is the only point of $S$ in the image;
ignoring basepoints, this induces 
\[
\Z \isom \pi_1(D^\times) \to \pi_1(\PP^1(\C)-S) \surjects G
\]
sending $1$ to a generator of the local monodromy group (inertia group) at $s$.
Changing basepoints conjugates the generator by an element of $G$,
so we obtain a well-defined conjugacy class for each $s$, hence a map $\cc \colon S \to \calC$, 
which we call the \defi{local monodromy}.

Algebraically, if $u$ is an analytic uniformizer at $s \in \PP^1(\C)$, 
the completion of $\C(\PP^1)$ at $s$ is $\C((u))$, and we have analogous homomorphisms involving \'etale fundamental groups
\[
\Zhat(1) \isom \pi_1^{\et}(\Spec \C((u))) \to \pi_1^{\et}(\PP^1_\C-S) \surjects G,
\]
factoring through $\mu_n$, defined up to conjugation by elements of $G$.

Thus we get a canonical map
\[
   \cc \colon S \to \Hom(\mu_n,G)/\Inn(G) \isom \Hom(\Z/n\Z,G)/\Inn(G) \isom G/\Inn(G) = \calC,
\]
depending only on a choice of isomorphism $\mu_n \isom \Z/n\Z$; we fix $e^{2\pi i/n} \mapsto 1$.

If $Y_\C \to \PP^1_\C$ is the base change of a $G$-cover $Y \to \PP^1_\Q$,
then $S$ may be viewed as a finite \'etale subscheme of $\PP^1_\Q$,
and $\cc \colon S(\Qbar) \to \calC$ is $\GG_\Q$-equivariant.
What we just described is Fried's branch cycle argument \cite[pp.~62--64]{Fried1977}.

\subsection{Rigidity}
\label{S:rigidity}

See \cite{Fried-Voelklein1991}, especially Corollary~1, for details on the theory in this section.
Fix a finite \'etale subscheme $S \subset \PP^1_\Q$ and a $\GG_\Q$-equivariant map $\cc \colon S(\Qbar) \to \calC$. 
Write $S(\Qbar) = \{s_1,\ldots,s_r\}$, and $C_i \colonequals \cc(s_i)$ for $i=1,\ldots,r$.

Let $\calH_\cc$ be the $0$-dimensional Hurwitz scheme over $\Q$ parameterizing $G$-covers of $\PP^1$ unramified outside $S$ with local monodromy $\cc$.
Because $Z(G)=1$, $\calH_\cc$ is a fine moduli space \cite[Proposition~1.4(b)]{Coombes-Harbater1985}.
Define 
\begin{align*}
   \bar{\Sigma}_\cc &\colonequals \{(g_1,\ldots,g_r) \in C_1 \times \cdots \times C_r : g_1 \cdots g_r = 1 \}, \\
   \Sigma_\cc &\colonequals \{(g_1,\ldots,g_r) \in \bar{\Sigma}_\cc : \textup{$g_1,\ldots,g_r$ generate $G$}\}, \\
   \Ni_\cc &\colonequals \Sigma_\cc/\Inn(G);
\end{align*}
here $g \in G \isom \Inn(G)$ acts on $r$-tuples by simultaneous conjugation. 
The set $\Ni_\cc$ is called the \defi{Nielsen class} of $\cc$.
By the Riemann existence theorem, $\calH_\cc(\C)$ is in bijection with $\Ni_\cc$. 

The size $\lvert \bar{\Sigma}_\cc \rvert$ can be computed from the character table of $G$; an inclusion--exclusion gives $\lvert \Sigma_\cc \rvert$; since the group $\Inn(G) \isom G$ acts freely on $\Sigma_\cc$, dividing by $|G|$ gives $\lvert \Ni_\cc \rvert$, which depends only on the multiset $\{C_1,\ldots,C_r\}$; see \cite[\S7.2 and \S7.3]{Serre1992} for an exposition.
When the branch locus is invariant under complex conjugation, 
Hurwitz described the action of complex conjugation on $\Ni_\cc$ in group-theoretic terms \cite[Theorem~1.1.2]{Malle-Matzat2018}.
The action of $\GG_\Q$ is more mysterious.

We restrict to the case $r=3$.
If $\lvert \Ni_\cc \rvert = 1$, then the triple $(C_1,C_2,C_3)$ is called \defi{rigid};
in that case, $\calH_\cc \isom \Spec \Q$ and the unique point corresponds to a $G$-cover $Y \to \PP^1_\Q$ with local monodromy $\cc$.

\section{The Mathieu group \texorpdfstring{$M_{23}$}{M23}}
\label{S:M23}

\subsection{A Nielsen class}

The Mathieu group $M_{23}$ is a finite simple group of order $10{,}200{,}960 = 2^7 \cdot 3^2 \cdot 5 \cdot 7 \cdot 11 \cdot 23$.
It is the automorphism group of the unique Steiner system $S(4,7,23)$.
It is a $4$-transitive subgroup of $S_{23}$, generated by the permutations $g_1,g_2$ below.
It has 17 conjugacy classes:
\begin{center}
  1, \; 2, \; 3, \; 4, \; 5, \; 6, \; 7A, 7B, \; 8, \; 11A, 11B, \; 14A, 14B, \; 15A, 15B, \; 23A, 23B
\end{center}
(we drop the A in the label when there is a unique conjugacy class of elements of a particular order).
We calculate $\lvert \Ni_\cc \rvert$ for all 86 $\GG_\Q$-stable $3$-element multisets of nontrivial classes in $\calC$,
using character tables as in Section~\ref{S:rigidity}.
We find that $\lvert \Ni_\cc \rvert = 0$ for $30$ of these, but
the value of $\lvert \Ni_\cc \rvert$ is never $1$, so the rigidity method fails.

We next investigate the multiset for which $\lvert \Ni_\cc \rvert$ has smallest positive size.
This is $\{ 2, \text{23A}, \text{23B} \}$, with $\lvert \Ni_\cc \rvert = 7$.
In fact, $M_{23}$ is small enough that a computer can quickly list triples representing all the elements of $\Ni_\cc$.
We fix the following such triple of permutations in $S_{23}$:\footnote{We follow Magma's convention that $S_{23}$ acts on the right on $\{1,\ldots,23\}$.  For example, applying $g_1 g_2$ to $5$ yields $22$, not $14$.}
\begin{align*}
   g_1 &\colonequals (1,11)(2,23)(3,8)(4,16)(5,21)(7,20)(15,19)(18,22) , \\
   g_2 &\colonequals (1,2,3,4,10,11,12,7,19,18,8,6,9,16,17,21,22,5,14,20,13,15,23), \\
   g_3 &\colonequals (1,2,11,10,16,9,6,3,23,19,20,14,21,17,4,8,22,5,18,15,13,7,12). 
\end{align*}

Let $S$ be the finite \'etale $\Q$-scheme whose geometric points are $0,\sqrt{-23},-\sqrt{-23}$.
Let $\cc \colon S(\Qbar) \to \calC$ send these points to classes 2, 23A, 23B, respectively.
This $\cc$ is $\GG_\Q$-equivariant.
We have $(g_1,g_2,g_3) \in \Sigma_\cc$, representing a class in $\Ni_\cc$.
By \cite[Theorem~1.1.2]{Malle-Matzat2018}, 
the map $(g_1,g_2,g_3) \mapsto (g_1^{-1}, g_3^{-1}, g_2^{-1})$ on $\Sigma_\cc$ induces complex conjugation on $\Ni_\cc$;  thus we compute that the $\Gal(\C/\R)$-orbits in $\Ni_\cc$ have sizes $1,1,1,2,2$, though we will not need this fact. 

Let $Y_\C \stackrel{\phi}\to \PP^1_\C$ be the corresponding $M_{23}$-cover.
Let $M_{22} \colonequals \Stab_{M_{23}}(1)$, where $M_{23} \le S_{23}$ acts on $\{1,\ldots,23\}$.
Let $X_\C$ be the (variety) quotient $Y_\C / M_{22}$,
so $\phi$ factors as $Y_\C \to X_\C \stackrel{t}\to \PP^1_\C$.
Then $X_\C \stackrel{t}\to \PP^1_\C$ is a \emph{non-Galois} degree~$23$ cover branched over $S(\C)$,
and its Galois closure is $Y_\C \to \PP^1_\C$.
Since $\lvert \Ni_\cc \rvert = 7$, and $\calH_\cc$ is a fine moduli space, 
the cover $X_\C \to \PP^1_\C$ descends to a cover $X_E \to \PP^1_E$ defined over a number field $E$ of degree $\le 7$ over $\Q$.

\begin{lemma}
\label{L:genus 4}
    The genus $g_X$ of $X_\C$ is $4$.
\end{lemma}

\begin{proof}
The ramification indices at the ramification points of $t \colon X_\C \to \PP^1_\C$ are the lengths of the cycles in $g_1,g_2,g_3$, so the degree of the ramification divisor is $8(2-1) + (23-1)+(23-1)=52$.
The Riemann--Hurwitz formula gives $2g_X-2 = 23(-2) + 52$, so $g_X=4$.
\end{proof}

\subsection{A numerical Belyi map}
\label{S:Belyi}

Model the hyperbolic plane by the open unit disk $D$. Let $q$ denote the standard coordinate on $D$, and let $d(z,w)$ denote the hyperbolic distance between $z,w\in D$.
Given $\theta \in (0,\pi/4)$, define
\[
    a \colonequals e^{-i \theta} \sqrt{\tan\left(\frac{\pi}{4}-\theta\right)}, \qquad b \colonequals 0, \qquad c \colonequals \sqrt{\cos 2\theta}.
\]

\begin{lemma}
  The points $a,b,c$ are the vertices of a hyperbolic triangle in $D$ with corresponding interior angles $\pi/2$, $\theta$, $\theta$. 
\end{lemma}

\begin{proof}
Note that $|a|,|b|,|c|< 1$, and $a$ does not lie on the geodesic through $b$ and $c$, so $a,b,c$ are the vertices of a hyperbolic triangle $T_{abc}$ in $D$.
Lengths of two of its sides satisfy
\[
\cosh d(b,a) = \frac{1+|a|^2}{1-|a|^2} = \frac{1+\frac{1-\tan \theta}{1+\tan \theta}}{1-\frac{1-\tan \theta}{1+\tan \theta}} = \cot \theta,
\quad 
\cosh d(b,c) = \frac{1+|c|^2}{1-|c|^2} = \frac{1+\cos 2\theta}{1-\cos 2 \theta} = \cot^2 \theta.
\]
On the other hand, since $\pi/2+\theta+\theta<\pi$, there exists a hyperbolic triangle $T_{\pi/2,\theta,\theta}$ with angles $\pi/2,\theta,\theta$ \cite[Theorem~7.16.2]{Beardon1983}.
By \cite[Theorem~7.11.3]{Beardon1983}, its leg lengths $\ell$ and hypotenuse $h$ satisfy $\cosh \ell = \cot \theta$ and $\cosh h = \cot^2 \theta$.
Moreover, $\angle b = \arg(c) - \arg(a) = \theta$.
By side-angle-side, $T_{abc} \cong T_{\pi/2,\theta,\theta}$.
Thus $\angle a = \pi/2$ and $\angle c = \theta$.
\end{proof}

Set $\theta \colonequals \pi/23$.
Let $\Delta$ be the triangle group generated by counterclockwise hyperbolic rotations by angles $2\pi/2$, $2\pi/23$, $2\pi/23$ around $a,b,c$: 
\[
    \Delta \colonequals \langle \, \delta_a,\delta_b,\delta_c \mid \delta_a^2=\delta_b^{23}=\delta_c^{23}=\delta_a\delta_b\delta_c=1 \, \rangle.\footnote{The papers \cite{Klug-Musty-Schiavone-Voight2014, Musty-Schiavone-Sijsling-Voight2019} write $\delta_c \delta_b \delta_a=1$ instead, because their convention is to let $\Delta$ act on the left.}
\]
Let $\psi \colon \Delta \to M_{23}$ be the surjection sending $\delta_a,\delta_b,\delta_c$ to $g_1,g_2,g_3$, respectively.
Let $\Gamma \colonequals \psi^{-1}(M_{22}) \le \Delta$.
Then $D/\Gamma \to D/\Delta$ is the desired cover 
$X_\C \stackrel{t}\to \PP^1_\C$ with monodromy given by $(g_1,g_2,g_3)$.
Let $a',b',c'$ be the images of $a,b,c$ under $D \surjects X_\C$.

Let $E' \colonequals E(\sqrt{-23}) \subset \C$.
Let $X_{E'} \colonequals X_E \times_E E'$.

\begin{lemma}
\label{L:b'}
We have $b' \in X_E(E')$, and $q$ corresponds to an analytic uniformizer at $b'$.
\end{lemma}

\begin{proof}
The point $b' \in X_\C(\C)$ is the unique geometric point mapping to $\sqrt{-23} \in \PP^1(E') \subset \PP^1(\C)$ under $X_{E'} \to \PP^1_{E'}$, so $b' \in X_E(E')$.
The map $D \surjects X_\C$ is unramified at $b$ since the ramification indices of $b$ and $b'$ over $\PP^1$ are both $23$, so the analytic uniformizer $q$ at $b$ corresponds to an analytic uniformizer at $b'$.
\end{proof}

A basis of $\HH^0(X_\C,\Omega^1)$ pulls back to a $4$-tuple $(f_0 \, dq,\ldots,f_3 \, dq)$ for some $f_0,\ldots,f_3 \in \C[[q]]$.
For any field $k$ and $r,d \in \Z_{\ge 0}$, let $k[x_0,\ldots,x_r]_d$ be the space of degree~$d$ homogeneous polynomials in $k[x_0,\ldots,x_r]$.
If $d \le 3$ and $R \in \C[x_0,\ldots,x_3]_d$ and $R(f_0,\ldots,f_3) \equiv 0 \pmod{q^{19}}$,
then $R(f_0,\ldots,f_3)=0$,
since the number of zeros of a nonzero global section of $(\Omega^1)^{\tensor d}$ is at most $\deg ((\Omega^1)^{\tensor d}) = (2g_X-2)d \le 18$.

Given $(g_1,g_2,g_3)$ and $N \ge 1$, the algorithm in \cite{Klug-Musty-Schiavone-Voight2014, Musty-Schiavone-Sijsling-Voight2019}, inspired by \cite{Hejhal1999} and \cite{Stark1984}, 
computes the first $N$ coefficients of each $f_i$ numerically.
The BelyiDB Magma package \cite{BelyiDB} implements this algorithm.
We run it with $N=19$ and compute $40$ digits of precision for each coefficient.  
Below, we work numerically without controlling errors, so the intermediate steps of the calculation are not rigorous, but we return to rigor at the end by verifying \emph{algebraically} that the claimed polynomials define $M_{23}$-extensions of $\Q$ and a regular $M_{23}$-extension of $\Q(t)$. 

We echelonize the basis.
This calculation shows that no nonzero $1$-form vanishes at $b'$ to order greater than $3$; that is, $b'$ is not a Weierstrass point.
In other words, we obtain $f_i = q^i + O(q^4)$. 

\begin{lemma} 
The genus~$4$ curve $X_\C$ is non-hyperelliptic.
\end{lemma}

\begin{proof}
By \cite[Lemma~4.3]{Bommel-Costa-Poonen-Srinivasan2026}, $\dim \{P \in \C[x_0,\ldots,x_3]_2 : P(f_0,\ldots,f_3)=0\}$ is $\binom{4-1}{2} = 3$ if $X_\C$ is hyperelliptic, and $\binom{4-2}{2} = 1$ otherwise.
We compute that the dimension is $1$.
\end{proof}

The canonical model of a non-hyperelliptic curve of genus~$4$ over a field $k$ is a complete intersection $P=Q=0$ in $\PP^3$, where $P \in k[x_0,\ldots,x_3]_2$ and $Q \in k[x_0,\ldots,x_3]_3$ \cite[Example~IV.5.2.2]{Hartshorne1977}. 
This curve in $\PP^3$ does not determine $P$ and $Q$ uniquely: one can replace $(P,Q)$ by $(\epsilon P ,\epsilon' Q + \ell P)$ for any nonzero scalars $\epsilon,\epsilon'$ and linear form $\ell$.

By numerical linear algebra (singular value decomposition), we compute  
$P_1 \in \C[x_0,\ldots,x_3]_2$ and $Q_1 \in \C[x_0,\ldots,x_3]_3 - \C[x_0,\ldots,x_3]_1 P_1$ such that $P_1(f_0,\ldots,f_3) = Q_1(f_0,\ldots,f_3) = 0$.
Let $X_1$ be the curve $P_1 = Q_1 = 0$ in $\PP^3_\C$.

\subsection{Recognition of algebraic coefficients}
\label{S:recognition}

Given Lemma~\ref{L:b'}, one might hope that
\begin{itemize}
\item $X_1$ is defined over $E'$ in the sense that its homogeneous ideal $I \subset \C[x_0,\ldots,x_3]$ is generated by polynomials in $E'[x_0,\ldots,x_3]$; and
\item $P_1,Q_1 \in E'[x_0,\ldots,x_3]$.
\end{itemize}
But there are issues:
\begin{enumerate}[\hspace{1em}\upshape (i)]
\item We expanded $1$-forms using an analytic uniformizer $q$ over $\C$, not $E'$, so we should not expect $X_1$ to be defined over $E'$.
\item Even if $X_1$ is defined over $E'$, our $P_1,Q_1$ might not be, because $P_1,Q_1$ are not uniquely determined by $X_1$.
\end{enumerate}
To address~(i), for suitable $\lambda \in \C^\times$ we will re-echelonize the basis using the scaled parameter $q' \colonequals \lambda q$ instead of $q$.
This amounts to making the substitution $x_i \colonequals \lambda^{-i} x_i'$ for $i=0,\ldots,3$ in $P_1,Q_1$ to produce new polynomials $P_\lambda,Q_\lambda$; we then rename $x_i'$ as $x_i$.
Define the weight of a monomial by setting $\wt(x_i) \colonequals i$, so that each weight~$w$ monomial scales by $\lambda^{-w}$.
Let $X_\lambda$ be the curve $P_\lambda=Q_\lambda=0$ in $\PP^3_\C$.

\begin{lemma}
\label{L:scaling}
The set of $\lambda \in \C^\times$ such that $P_\lambda \in \C^\times E'[x_0,\ldots,x_3]_2$
is a coset $C$ of $(E')^\times$ in $\C^\times$.
For any $\lambda \in C$, we have $Q_\lambda \in \C^\times E'[x_0,\ldots,x_3]_3 + \C[x_0,\ldots,x_3]_1 P_\lambda$.
\end{lemma}

\begin{proof}
The computed coefficients of $x_0 x_2$ and $x_0 x_3$ in $P_1$ are nonzero.
The set of $\lambda \in \C^\times$ such that the ratio of these coefficients in $P_\lambda$ is in $E'$ is a coset of $(E')^\times$, since $\lambda$ scales the ratio by $\lambda^{-2}/\lambda^{-3} = \lambda$.
It remains to show that there exists $\lambda$ such that $X_\lambda$ is defined over $E'$,
since then $P_\lambda \in \C^\times E'[x_0,\ldots,x_3]_2$
and $Q_\lambda \in \C^\times E'[x_0,\ldots,x_3]_3 + \C[x_0,\ldots,x_3]_1 P_\lambda$.

There is an analytic uniformizer $\tau$ at $b' \in X_{E'}$ such that $\tau^{23}$ is an analytic uniformizer at $\sqrt{-23} \in \PP^1_{E'}$.
Then $\tau = \sum_{n \ge 1} \lambda_n q^n$ for some $\lambda_1,\lambda_2,\ldots \in \C$ with $\lambda_1 \ne 0$.
Let $\zeta = e^{2\pi i/23}$.
The rotation $\delta_b$ maps $q$ to $\zeta q$;
compatibly, the local monodromy generator $g_2$ at $b'$ maps $\tau$ to $\zeta \tau$.
Thus $\zeta \tau = \sum_{n \ge 1} \lambda_n (\zeta q)^n$.
For each $n \ge 1$, comparing coefficients of $q^n$ gives $\zeta \lambda_n = \lambda_n \zeta^n$, so $\lambda_n = 0$ when $n \not\equiv 1 \pmod{23}$.
Thus $\tau = \lambda q + O(q^{24})$, where $\lambda \colonequals \lambda_1$.

Echelonizing using $\tau$ produces a curve $\calP = \calQ = 0$ in $\PP^3_{E'}$
since $\tau$ is a uniformizer at an $E'$-point of a curve defined over $E'$.
Echelonizing using $\lambda q$ produces $X_\lambda$.  These curves are the \emph{same} since the construction of $\calP$ and $\calQ$ examines only the first $19$ coefficients of each expansion. 
Thus $X_\lambda$ is defined over $E'$.
\end{proof}

Choose $\lambda$ so that the ratio of the coefficients of $x_0 x_2$ and $x_0 x_3$ in $P_\lambda$ is $1$.
Then scale $P_\lambda$ to produce a polynomial $P_\lambda^1$ whose coefficient of $x_0 x_2$ is $1$.
A linear algebra computation shows that there is a unique $\ell \in \C[x_0,\ldots,x_3]_1$ such that the coefficients of $x_0^2 x_2$, $x_0 x_1 x_2$, $x_1^2 x_2$, $x_1^2 x_3$ in $Q_\lambda - \ell P_{\lambda}^1$ are $0$;
then scale to produce a polynomial $Q_\lambda^1$ whose coefficient of $x_0^2 x_3$ is $1$.
By Lemma~\ref{L:scaling}, $P_\lambda^1, Q_\lambda^1 \in E'[x_0,\ldots,x_3]$.

Lattice reduction \cite{LLL} lets us recognize algebraic numbers of small degree and height 
from numerical approximations; 
our precision turns out to be sufficient to recognize the coefficients of $P_\lambda^1$, but not $Q_\lambda^1$.
It would be expensive to re-run BelyiDB with greater precision to recognize $Q_\lambda^1$, so instead we will find a better $\lambda \in C$ to reduce the height.
Let $\calO$ be the ring of integers of $E'$.

\begin{lemma}
\label{L:fractional ideal}
The set of $\alpha \in (E')^\times$ such that $P_{\lambda/\alpha}^1 \in \calO[x_0,\ldots,x_3]$,
together with $0$, is a fractional ideal $\mathfrak{a}$ of $E'$.
\end{lemma}

\begin{proof}
Replacing $\lambda$ by $\lambda/\alpha$ multiplies each weight~$w$ monomial in $P_\lambda$ by $\alpha^w$,
and in $P_\lambda^1$ by $\alpha^{w-2}$.
Examining the coefficient of $q^2$ in $P_1(f_0,\ldots,f_3)$ shows that the coefficients of the weight~$2$ monomials $x_0 x_2$ and $x_1^2$ in $P_1$ are negatives of each other, so these coefficients in $P_\lambda^1$ are $1$ and $-1$.
Each other monomial appearing in $P_\lambda^1$ has weight $w>2$, so the condition that its coefficient in $P_{\lambda/\alpha}^1$ lie in $\calO$ is that $\alpha^{w-2}$ belongs to some fractional ideal.
Therefore, the set in question is an intersection of finitely many fractional ideals, hence a fractional ideal.
\end{proof}

Using the exact coefficients of $P_\lambda^1$, we compute the fractional ideal $\mathfrak{a}$ of Lemma~\ref{L:fractional ideal}.
It might not be principal, but we choose $\alpha \in \mathfrak{a} - \{0\}$ minimizing $\sup \{ \lvert \iota(\alpha) \rvert : \iota \in \Hom(K,\C) \}$ to reduce the coefficient heights.
Now it turns out that $P \colonequals P_{\lambda/\alpha}^1$ and $Q \colonequals Q_{\lambda/\alpha}^1$ are of small enough height that we succeed in recognizing their coefficients in $E'$.

Miraculously, they lie in $K = \Q(\sqrt{-23})$, suggesting that $X_\C$ might descend to $\Q$! 
The exact $P$ and $Q$ define a canonically embedded smooth genus~$4$ curve $X_K \subset \PP^3_K$.
Evaluating the renormalized $(f_0:\cdots:f_3)$ at $a,b,c$ gives $a',b',c' \in X_K(\C) \subset \PP^3(\C)$
with $b' = (1:0:0:0)$, and lattice reduction recognizes $c' \in X_K(K) \subset \PP^3(K)$. 

\subsection{Canonical equation for the degree 23 cover}
\label{S:canonical}

In describing our computations in this section, we assume that there exists a degree~$23$ cover $X \stackrel{t}\to \PP^1_\Q$ with Galois group $M_{23}$ and local monodromy (2,23A,23B) above $S$ whose base change to $K$ is $X_K$, such that $\Gal(K/\Q)$ interchanges $b'$ and $c'$.
This descent assumption is not yet justified, but later we will verify independently that the computed cover $X \stackrel{t}\to \PP^1_\Q$ has the correct Galois group and local monodromy.

Specifically, given the exact $P,Q,b',c'$ over $K$,
we will describe rational functions $t,v$ generating the function field $\Q(X)$,
and the output of this section will be an irreducible polynomial $F \in \Q[T,V]$ such that $F(t,v)=0$.
\smallskip

\emph{Construction of $t$}: 
We first construct 
$X_K \stackrel{u}\to \PP^1_K$ ramified over $1$, $0$, $\infty$ instead of $0$, $\sqrt{-23}$, $-\sqrt{-23}$,
with ramification index $23$ above $0$ and $\infty$.
A Riemann--Roch space computation produces $u \in K(X_K)$ such that $\divv(u) = 23b' - 23c'$.
The degree~$23$ map $u \colon X_K \to \PP^1_K$ is ramified above $0$, $\infty$, and one other point above which the fiber is $D_7+2D_8$ for some reduced effective divisors $D_7$ and $D_8$ of degrees $7$ and $8$, respectively.
Then $\divv(du/u) = D_8 - b' - c'$ (there can be no more zeros, since $\deg \divv(du/u) = 2g_X-2 = 6$).
We compute the zeros of $du/u$, and scale $u$ so that its value on them is $1$.
Let $\rho \in \Aut(\PP^1_K)$ send $1,0,\infty$ to $0, \sqrt{-23}, -\sqrt{-23}$, respectively.
We compute $t = \rho \circ u$.
\smallskip

\emph{Construction of $v$}: 
We compute that the $K$-vector space $L(D_8-3b'-2c')$ is $1$-dimensional.
Compute a generator $f$.
Computation shows that $\sum_{x \in D_7} f(x) \ne 0$, where the sum is over the $7$ geometric points of $D_7$.
Normalize $f$ to make this sum $1$.
Likewise, compute a generator $g \in L(D_8-2b'-3c')$ such that $\sum_{x \in D_7} g(x)=1$.
Then $\{f,g\}$ is a $K$-basis for the linear system $L(D_8-2b'-2c')$.
For the intended $\Q$-structure on $X_K$, the generator of $\Gal(K/\Q)$ fixes $D_8$ and interchanges $b',c'$ and interchanges $f,g$.
Let $v = \sqrt{-23}(f-g)/(f+g)$, which is a rational function of degree~$4$.
\smallskip

\emph{Construction of $F$}:
If $F \in \Q[T,V]$ has $\deg_T F = m$ and $\deg_V F = n$, say that $F$ has bidegree $(m,n)$.
By linear algebra, we compute the bidegree $(4,23)$ polynomial $F_0 \in \Q[T,V]$ vanishing at $(t,v)$ whose coefficient of $T^4 V^{23}$ is $1$.
Actually, to avoid coefficient growth in the calculations,
we compute $t,f,g,v,F_0$ modulo many primes $p$ of size about $10^7$
which split in $K$ and reconstruct the rational coefficients of $F_0$.
Because we have no a priori bound on the numerators and denominators of the rational coefficients,
the reconstruction could be wrong, but we will verify the result later.
Finally, for aesthetic purposes, we clear denominators to get $F \in \Z[T,V]$.

\begin{remark}
Any rational function $v$ such that $\Q(t,v) = \Q(X)$, or the polynomial relation between $t$ and $v$, can be used to prove Theorems~\ref{T:over Q} and \ref{T:over Q(t)}.
In the first version of this paper, we used a relation $F_8 \in \Q[T,V]$ of bidegree $(8,23)$ instead of $(4,23)$.
There we also gave a theoretical description of one of bidegree $(4,23)$.
Le~Lay \cite[Theorem~1.1]{LeLay-preprint} was the first to compute an explicit polynomial of bidegree $(4,23)$;
he simplified our $F_8$ model by using an adjoint calculation.
Our construction of $F$ of bidegree $(4,23)$ does not rely on \cite{LeLay-preprint} and does not use $F_8$.
\end{remark}

\begin{remark}
In fact, degree~$4$ in $T$ is best possible.
The only nonconstant rational maps $X_\C \to \PP^1_\C$ of degree $\le 3$ come from the rulings on the quadric $P=0$ \cite[Example~IV.5.5.2]{Hartshorne1977}.
Since the discriminant of $P$ is not a square in $K$,
the rulings are not defined over $K$, let alone over $\Q$.
\end{remark}

\begin{remark}
\label{R:change of variable}
The largest coefficient of $F$ has $63$ digits. 
By substituting $V = \frac{37W-40}{49W+32}$, 
clearing denominators, and dividing by the content, we produce a polynomial $F_{\sma} \in \Z[T,W]$ of bidegree $(4,23)$ whose largest coefficient has only $21$ digits.
\end{remark}

\subsection{Verification of ramification}

Now that we have a candidate $F(T,V)=0$, we need to verify that it produces an $M_{23}$-cover of $\PP^1_\Q$ with the correct local monodromy.
Let $X_0$ be the closure of $F(T,V)=0$ in $\PP^1_\Q \times \PP^1_\Q$.

\begin{proposition}
\label{P:62 nodes}
The singular locus of $X_0$ consists of $62$ geometric points, each a node,
and their $T$-coordinates are distinct.
\end{proposition}

\begin{proof}
We calculate $\disc_V F = e T^8(T^2+23)^{22} h_{62}^2$ 
for some nonzero $e \in \Z$ and degree~$62$ irreducible polynomial $h_{62} \in \Z[T]$.
Thus $X_0$ is smooth except possibly at points where $T=\infty$, $T=0$, $T^2+23=0$, or $h_{62}=0$.
Above a root $\beta$ of $h_{62}$, we verify that $F(\beta,V)$ has $22$ distinct roots, with one being a double root; since the valuation at $\beta$ of $\disc_V F$ is $2$, this implies that the double root is a node.
We check partial derivatives at the points above $T=\infty$, $T=0$, $T^2+23=0$.
\end{proof}

Let $X$ be the normalization of $X_0$.
Let $t$ be the projection $X \surjects X_0 \stackrel{T}\to \PP^1$.

\begin{proposition}
\label{P:ramification}
The curve $X$ is geometrically irreducible.
The morphism $X \stackrel{t}\to \PP^1_\Q$ is unramified outside $S$.
The local monodromy above each of $\pm \sqrt{-23}$ is a $23$-cycle,
and above $0$ it has cycle type $1^7 2^8$.
\end{proposition}

\begin{proof}
Ramification can occur at most above $\infty$ and zeros of $\disc_V F$.
Above $h_{62}=0$, the curve $X_0$ has only nodes, so $X \to \PP^1_\Q$ is unramified there.
Above $\infty$, $0$, $\pm \sqrt{-23}$, the fibers of $X$ and $X_0$ agree, and we compute the latter.
Since $X \to \PP^1_\Q$ is totally ramified above $\pm \sqrt{-23}$, it is geometrically irreducible.
\end{proof}

The genus of $X$ is $(4-1)(23-1) - 62 = 4$.
Let $Y \to \PP^1_\Q$ be the Galois closure of the degree~$23$ morphism $X \stackrel{t}\to \PP^1_\Q$.  
We will prove that $Y \to \PP^1_\Q$ is an $M_{23}$-cover.

\subsection{Verification of Galois groups}
\label{S:verification}

Let $G_\geom$ and $G_\arith$ be the geometric and arithmetic monodromy groups of $F(t,V) \in \Q(t)[V]$.
Thus $G_\geom \lhd G_\arith \le S_{23}$.

\begin{proposition} \label{PgeoM}
We have $G_\geom = M_{23}$. 
\end{proposition}

\begin{proof} 
By the local monodromy in Proposition~\ref{P:ramification} and the classification of transitive groups of degree $23$, we have $M_{23} \subseteq G_{\geom}$.

Recall the factorization $\disc_V F = e t^8 (t^2+23)^{22} h_{62}^2$.
Let $S'$ be the subscheme of $\Aff^1_\Z \subset \PP^1_\Z$ defined by $t(t^2+23)h_{62}=0$.
To calculate $G_{\geom}$, we reduce modulo $p=411000011$; see Appendix~\ref{A:rigorous} for the reason.
We check that $p \nmid e$, and $S'$ remains \'etale modulo $p$.
Let $G_{\geom,p}$ and $G_{\arith,p}$ be the geometric and arithmetic monodromy groups of $F \bmod p$.
The specialization map on prime-to-$p$ fundamental groups of $\PP^1 - S'$ is an isomorphism (\cite[XIII.2]{SGA1}; see also \cite[Th\'eor\`eme~4.4]{Orgogozo-Vidal2000}), so $G_{\geom,p}=G_{\geom}$.
Using Appendix~\ref{A:rigorous}, we certify that $G_{\arith,p} = M_{23}$.
Thus $M_{23} \subseteq G_{\geom} = G_{\geom,p} \subseteq G_{\arith,p} = M_{23}$,
so all are equal.
\end{proof}

\begin{corollary}
\label{C:G_arith}
We have $G_\arith=M_{23}$.
\end{corollary}

\begin{proof}
By Proposition~\ref{PgeoM}, $G_\geom = M_{23}$.
The group $G_\geom$ is normal in $G_\arith$,
and $M_{23}$ is not normal in any larger subgroup of $S_{23}$ (i.e., it is self-normalizing in $S_{23}$).
\end{proof}

By Proposition~\ref{PgeoM} and Corollary~\ref{C:G_arith},
since $G_\geom = G_\arith$, the splitting field of $F$ over $\Q(t)$ is a regular extension
with Galois group $M_{23}$.  This proves Theorem~\ref{T:over Q(t)}.

\subsection{Galois action on the Nielsen class}
\label{S:Galois on Nielsen}

We now repeat the computations of Sections~\ref{S:Belyi}--\ref{S:verification} for a triple representing a \emph{different element} of the size~$7$ Nielsen class $\Ni_\cc$.
We use BelyiDB and recognize\footnote{Some modifications to the recognition procedure of Section~\ref{S:recognition} were necessary.  First, we ran BelyiDB with 60 digits of precision instead of just 40.  Second, we actually ran BelyiDB on all six other Nielsen class elements, in order to recover $P$ and $Q$ together with their conjugates simultaneously.  Third, it turned out that 60 digits of precision were not sufficient to recognize \emph{all} the coefficients of $P_\lambda^1$; we recognized only the coefficients of monomials of weight $\le 4$, and defined $\mathfrak{a}$ using only the conditions on those coefficients.  Then, after choosing a small $\alpha \in \mathfrak{a}$, we succeeded in recognizing all coefficients of $P_{\lambda/\alpha}^1$ and $Q_{\lambda/\alpha}^1$.} the coefficients of the new $P$ and $Q$ as elements of $L' \colonequals L(\sqrt{-23})$, where $L \colonequals \Q[x]/(x^6 - x^5 - 3x^4 - 3x^3 + x^2 + 5 x + 4)$.
Here $L$ is a degree~$6$ number field  
whose Galois closure $\widetilde{L}$ has Galois group $S_6$.
The discriminant of $L$ is $2^4 \cdot 11 \cdot 23^4$, 
divisible only by primes dividing $|M_{23}|$, as required by \cite[Theorem~5.5]{Beckmann}.
The signature of $L$ is $(2,2)$, compatible with the $\Gal(\C/\R)$-action on $\Ni_\cc$.
Since $L$ has real places, $\sqrt{-23} \notin L$.
We compute a bidegree~(4,23) polynomial $\calF \in L[T,V]$.
We check that it defines a degree~$23$ cover $\calX \to \PP^1_L$ ramified only over $S$, with the correct ramification indices. 
We reduce modulo the degree~$1$ prime above $p=411000011$,
and use Appendix~\ref{A:rigorous} to verify that it has Galois group $M_{23}$.
Because $\sqrt{-23} \notin L$, the local monodromy is $\GG_L$-stable, and hence can only be (2,23A,23B) or (2,23B,23A).
In the latter case, compose with the involution of $\PP^1_L$ sending $0,\sqrt{-23},-\sqrt{-23}$ to $0,-\sqrt{-23},\sqrt{-23}$, respectively, to switch the 23A and 23B.

The cover $\calX \to \PP^1_L$ defines a point of $\calH_\cc(L)$, but a priori, it could be defined over any \emph{subfield} of $L$.
Since $\Gal(\widetilde{L}/\Q) \isom S_6$, the only subfields are $\Q$ and $L$.
The following rules out $\Q$.

\begin{proposition}
The curve $\calX$ does not descend to $\Q$.
\end{proposition}

\begin{proof}
For a prime ideal $\pp$ of $L$, let $\F_\pp$ be the residue field.
For two of the three degree~$1$ prime ideals above $53$, say $\pp_1$ and $\pp_2$, we compute $\#\calX(\F_{\pp_1}) = 45$ and $\#\calX(\F_{\pp_2}) = 52$.
If $\calX$ were defined over $\Q$, both would equal $\#\calX(\F_{53})$.
\end{proof}

Thus $\calH_\cc$ contains a closed point $\Spec L$.
On the other hand, the cover $X \to \PP^1_\Q$ we computed earlier shows that $\calH_\cc$ contains $\Spec \Q$.
Since $\deg \calH_\cc = 7$, we have $\calH_\cc = \Spec \Q \sqcup \Spec L$.
In particular, $\GG_\Q$ acts on $\Ni_\cc = \calH_\cc(\Qbar)$ with orbits of size $1$ and $6$.

\subsection{Automorphisms of the \texorpdfstring{$M_{23}$}{M23}-cover}

The topological Euler characteristic of the quotient stack $[Y/M_{23}]$ is
\[
    \chi([Y/M_{23}]) = 2 - (1-\tfrac{1}{2}) - (1-\tfrac{1}{23}) - (1-\tfrac{1}{23}) = -19/46.
\]
By the Riemann--Hurwitz formula, the genus $g_Y$ of $Y$ satisfies
\[
    2-2g_Y = \chi(Y) = \lvert M_{23} \rvert \; \chi([Y/M_{23}]),
\]
so $g_Y = 2{,}106{,}721$.

\begin{proposition}
\label{P:Aut}
We have $\Aut Y = \Aut Y_\C \isom M_{23}$.
\end{proposition}

\begin{proof}
First, here are two ways to see that $(\Aut Y_\C : M_{23}) \le 2$:
\begin{enumerate}[\upshape (i)]
\item Every finite-index overgroup of the triangle group $\Delta = \Delta(2,23,23)$ is another triangle group.  The classification of inclusions of triangle groups shows that the only nontrivial possibility is the index~$2$ overgroup $\Delta(2,4,23)$.  See~\cite[Theorem~1]{Singerman1972}.
\item If $n = (\Aut Y_\C : M_{23})$, then $\chi([Y_\C/\Aut Y_\C]) = (-19/46)/n$, but for $n \ge 3$, elementary number theory shows that $2 - \sum (1-\frac{1}{m_i}) = (-19/46)/n$ has no solutions in integers $m_i \ge 2$.
\end{enumerate}
If $(\Aut Y_\C : M_{23}) = 2$, then there would be an involution of $[Y/M_{23}]$ interchanging the points with local monodromy 23A and 23B, and $M_{23}$ would have an outer automorphism interchanging conjugacy classes 23A and 23B.  But all automorphisms of $M_{23}$ are inner.
\end{proof}

\begin{remark}
The proof of \cite[Theorem~II.9.9]{Malle-Matzat2018} finds an $M_{23}$-cover of $\PP^1_{\Q(\sqrt{-23})}$ that is also an $M_{24}$-cover of $\PP^1_{\Q(\sqrt{-23})}$.
Proposition~\ref{P:Aut} implies that our $M_{23}$-cover $Y$ cannot be an $M_{24}$-cover of $\PP^1_\Q$.
\end{remark}

\subsection{Specializations}

By specializing $T$ in $F(T,V)$ to an element $\qq \in \Q$, we obtain a degree~$23$ polynomial $f_\qq \in \Q[V]$ (the leading coefficient never vanishes).
Also define $f_\infty \in \Q[V]$ by setting $\tau=0$ in $\tau^4 F(1/\tau,V)$.
Let $K_\qq$ be the splitting field of $f_\qq$ over $\Q$.
Let $\Gal(f_\qq) = \Gal(K_\qq/\Q)$.

For another perspective, let $X_\qq$ be the fiber of $X \stackrel{t}\to \PP^1_\Q$ above $\qq \in \PP^1(\Q)$, 
so $X_\qq$ is a $\Q$-scheme.  
A computation shows that the singularities of the plane model $F(T,V)=0$ birational to $X$ do not lie above rational points of $\PP^1_\Q$, so the set of roots of $f_\qq$ is in bijection with $X_\qq(\Qbar)$, and the irreducible factors of $f_\qq$ correspond to the $\Gal(f_\qq)$-orbits in $X_\qq(\Qbar)$.
For $\qq \ne 0$ (not the rational branch point), we may identify $\Gal(f_\qq)$ with a subgroup of $M_{23}$, well defined up to $M_{23}$-conjugation.
We have $\Gal(f_\qq) \isom M_{23}$ for all $\qq$ outside a finite set: see Theorem~\ref{T:finite}.

Given $\qq \in \PP^1(\Q)-\{0\}$, we apply the PARI/GP routine \texttt{polredbest} \cite{Cohen1991, PARI} repeatedly to find lower-height monic polynomials in $\Z[x]$ defining the same number field as $f_\qq$,
and use \texttt{GaloisProof} to certify their Galois groups.
The specialization $\qq = 2$ gives in this way the polynomial in Example~\ref{Ex:M23 over Q}.
The specialization $\qq = -3703/947$ 
gives the following, which is the lowest-height $M_{23}$ polynomial we found so far.

\begin{example}
\label{Ex:M23 over Q again}
The splitting field of
\begin{align*} 
x^{23} &+ 46 x^{21} - 598 x^{20} + 1679 x^{19} - 21620 x^{18} + 127420 x^{17}\\ 
&- 361974 x^{16} + 2223732 x^{15} - 9392096 x^{14} + 17344116 x^{13}\\ 
&-71999476 x^{12} + 320807726 x^{11} - 436105484 x^{10} + 83587888 x^9\\ 
&-2463757240 x^8 + 9451874955 x^7 - 5728074376 x^6 - 26037806834 x^5\\ 
&+63691532334 x^4 - 67357061907 x^3 + 38754121124 x^2\\ 
&-11217790920 x + 1243077066 
\end{align*}
is an $M_{23}$-extension of $\Q$ unramified outside $\{2,7,23\}$.
\end{example}

\subsection{Other specializations}

In this section, we prove that there is a value of $\qq$ such that 
$\Gal(f_\qq) \isom M_{22}$, and prove Theorem~\ref{T:finite} which states that there are 
only finitely many values of $\qq$ such that $\Gal(f_\qq) \not \isom M_{23}$.

\begin{proposition}
\label{P:irreducible = M23}
\hfill
\begin{enumerate}[\upshape (a)]
\item \label{I:fiber above 0}
The fiber $X_0(\Qbar)$ of $X$ above the branch point $0$ consists of $\Gal(f_0)$-orbits of sizes $7$ and $8$, in which every point has ramification index $1$ or $2$, respectively. 
The action of $\Gal(f_0)$ on $X_0(\Qbar)$ is equivalent to $\AGL_3(\F_2)$ acting on the set $\check{\PP}^2(\F_2) \sqcup \Aff^3(\F_2)$ of size $7+8$, where $\check{\PP}^2$ denotes the space of \emph{lines} in the plane at infinity in $\PP^3$.\footnote{A version of this description of $\Gal(f_0)$ was independently computed in \cite[Section~3]{LeLay-preprint}, 
using an equation for $X \to \PP^1_\Q$ that we provided in the first version of this paper.}
\item \label{I:complex conjugation}
For each $\qq \in \PP^1(\Q)$, complex conjugation gives an order~$2$ element in $\Gal(f_\qq)$.
\item
For each $\qq \in \PP^1(\Q)$, 
\[
    \Gal(f_\qq) \isom M_{23} \quad \iff \quad \textup{$f_\qq$ is irreducible over $\Q$}.
\]
\end{enumerate}
\end{proposition}

\begin{proof}
\hfill
\begin{enumerate}[\upshape (a)]
\item 
As in the proof of Lemma~\ref{L:genus 4}, the fiber $X_0(\Qbar)$ consists of $7$ unramified points and $8$ points of ramification index~$2$, so $f_0 = g_7 g_8^2$ for some $g_7,g_8 \in \Q[x]$ of degree $7$ and $8$.
We use Magma to verify that $g_7$ and $g_8$ are irreducible, and \texttt{GaloisProof} to compute $\Gal(f_0)$ as a permutation group.

\item
The space $\PP^1(\R)-\{0\}$ is connected and contains no branch points of $t$,
so the conjugacy class of complex conjugation acting on $X_\qq(\C)$ is the same for all $\mathsf q \in \PP^1(\R)-\{0\}$; checking $\mathsf q=1$, say, shows that it is the class $2$, so $2$ divides $\lvert \Gal(f_\qq) \rvert$.
We check $\mathsf q=0$ separately.

\item We may assume that $\mathsf q \ne 0$.
If $f_\qq$ is reducible, then $\Gal(f_\qq)$ is contained in a product of symmetric groups of degree less than $23$, so $23$ does not divide $\lvert \Gal(f_\qq) \rvert$, so $\Gal(f_\qq) \not\isom M_{23}$.

Now suppose that $f_\qq$ is irreducible.
Then $\Gal(f_\qq)$ acts transitively on a set of size $23$, so $23$ divides $\lvert \Gal(f_\qq) \rvert$.
Also, $2$ divides $\lvert \Gal(f_\qq) \rvert$, by \eqref{I:complex conjugation}.
The only subgroup of $M_{23}$ of order divisible by both $2$ and $23$ is $M_{23}$ itself. 
\qedhere
\end{enumerate}
\end{proof}

\begin{proposition}
\label{X(Q)}
The genus~$4$ curve $X$ has a rational point.
\end{proposition}

\begin{proof}
A search produced $(-1545209/264829,-1/27)$ on the plane model $F(T,V)=0$.
\end{proof}

\begin{remark}
Under the change of coordinate in Remark~\ref{R:change of variable}, $V=-1/27$ corresponds to $W=1$.  
\end{remark}

\begin{proposition}
If $\qq=-1545209/264829$, then $\Gal(f_\qq) \isom M_{22}$.
\end{proposition}

\begin{proof}
We have $f_\qq=(V+1/27) h_{22}$,
where $h_{22} \in \Q[V]$ is of degree $22$.
We certify using \texttt{GaloisProof} that $h_{22}$ has Galois group $M_{22}$ over $\Q$.
\end{proof}

\begin{example}
Running \texttt{polredbest} produces a polynomial
\begin{align*}
&x^{22} - 10x^{21} - 171x^{20} + 4948x^{19} - 3160x^{18} - 878184x^{17} \\
&\quad {}+ 1612816x^{16} + 57499538x^{15} - 104404974x^{14} - 1334363150x^{13} \\
&\quad {}+ 222611762x^{12} + 13957937472x^{11} + 44873894586x^{10} + 48162643482x^{9} \\
&\quad {}- 396850048164x^{8} - 2898168029276x^{7} - 8595497273752x^{6} \\
&\quad {}- 12865403779026x^{5} - 5309597045808x^{4} + 12770814291390x^{3} \\
&\quad {}+ 19223398109406x^{2} + 8150987444096x + 212425572022
\end{align*}
defining the same number field as $h_{22}$.
Its splitting field is an $M_{22}$-extension of $\Q$.
\end{example}

\begin{conjecture}
\label{C:specializations}
For all $\qq \in \PP^1(\Q) - \{0,-1545209/264829\}$, we have $\Gal(f_\qq) \isom M_{23}$.
\end{conjecture}

The rest of this section collects theoretical and computational evidence for Conjecture~\ref{C:specializations}.

\begin{remark}
\label{R:sieve}
Write $\qq=a/b$ in lowest terms with $a,b \in \Z$ and $b \ne 0$.
If $a^2+23b^2$ is not of the form $2^i 3^j 23^k w^{23}$ for some $i,j,k,w \in \Z_{\ge 0}$, then $\Q[V]/(f_\qq)$ over $\Q$ is totally ramified of degree~$23$ above a prime $p \ne 2,3,23$ with $23 \nmid v_p(a^2+23b^2)$, so $\Gal(f_\qq) \isom M_{23}$.
Therefore we search through integers of this form (imposing other necessary local conditions at a few primes), 
use Cornacchia's algorithm \cite{Hardy1990deterministic,Basilla2004} to find representations of each as $a^2+23b^2$, 
and check for the resulting $\qq=a/b$ that $f_\qq$ is irreducible.
This proves that $\Gal(f_\qq) \isom M_{23}$ for all nonzero $\qq$ with $|a|,|b| \le 10^{50}$ except for $\qq=-1545209/264829$.
Moreover, $f_\infty$ is irreducible, so $\Gal(f_\infty) \isom M_{23}$ too.
\end{remark}

Table~\ref{tab:specializations} summarizes the information we computed about the specializations.
The second column gives the degrees of the irreducible factors of $f_\qq$ in $\Q[V]$;
they sum to $23$ except when $\qq=0$ (the rational branch point).

\medskip
\begin{table}[h]
\centering
\begin{tabular}{c|ccc}
\toprule
$\qq$ & irreducible factors of $f_\qq$ & $\Gal(f_\qq)$ & $\lvert \Gal(f_\qq) \rvert$ \\ 
\midrule
$0$ & $7 + 8$ & $\AGL_3(\F_2)$ & 1344 \\
$-1545209/264829$ & $1 + 22$ & $M_{22}$ & $443520$ \\
$\infty$, all other $a/b$ with $|a|,|b| \le 10^{50}$ & $23$ & $M_{23}$ & $10200960$ \\ 
\bottomrule
\end{tabular}
\bigskip
\caption{Specializations of the regular $M_{23}$-extension}
\label{tab:specializations}
\end{table}

\begin{theorem}
\label{T:finite}
The set of $\qq \in \PP^1(\Q)$ such that $\Gal(f_\qq) \not\isom M_{23}$ is finite.
\end{theorem}

\begin{proof}
Let $\qq \in \PP^1(\Q) - \{0\}$.
If $\Gal(f_\qq) \not\isom M_{23}$, then $\Gal(f_\qq)$ is contained in some maximal subgroup $H \le M_{23}$.
In this case, $\qq$ lifts to a rational point of the quotient curve $Y/H$.
Information about the maximal subgroups $H$ of $M_{23}$ can be found in the ATLAS
\cite{ATLAS} and \cite[Section~5.2.8]{WilsonFSG}.
We summarize this information in Table~\ref{tab:m23-maximal} using ATLAS notation.

For each $H$, intersecting $H$ with the classes 2, 23A, 23B gives the ramification of $Y \to Y/H$, and then the Riemann--Hurwitz formula gives the genus of $Y/H$.
Each quotient has genus greater than $1$, and hence only finitely many rational points \cite{Faltings1983}.
\end{proof}

\begin{table}[h]
\centering
\begin{tabular}{c|ccccc}
\toprule
$H$ & $|H|$ & $(M_{23}:H)$ & orbit sizes & genus of $Y/H$ & $\Q$-points on $Y/H$? \\
\midrule
$M_{22}$                      & 443520 & 23    & $1+22$  & 4    & yes, Proposition~\ref{X(Q)} \\
$L_3(4){:}2_2$                & 40320  & 253   & $2+21$  & 46   & maybe \\
$2^4{:}A_7$                   & 40320  & 253   & $7+16$  & 46   & yes, Proposition~\ref{P:frombranch} \\
$A_8$                         & 20160  & 506   & $8+15$  & 95   & maybe \\
$M_{11}$                      & 7920   & 1288  & $11+12$ & 253  & no, Remark~\ref{R:Q23 point} \\
$2^4{:}(3\times A_5){:}2$     & 5760   & 1771  & $3+20$  & 344  & maybe \\
$23{:}11$                     & 253    & 40320 & $23$    & 8327 & no, Remark~\ref{R:real point} \\ 
\bottomrule
\end{tabular}
\bigskip
\caption{Maximal subgroups $H \le M_{23}$ up to conjugacy, and quotient curves}
\label{tab:m23-maximal}
\end{table}

\begin{proposition}
\label{P:frombranch}
Let $H \le M_{23}$ be a maximal subgroup.
\begin{enumerate}[\upshape (a)]
\item
If $H \isom 2^4{:}A_7$, then $Y/H$ has a unique rational point above $\qq=0$.
\item
If $H \not\isom 2^4{:}A_7$, then $Y/H$ has no rational point above $\qq=0$.
\end{enumerate}
\end{proposition}

\begin{proof}
Let $Y_0$ be the fiber of $Y \to \PP^1$ above $0$.
Let $I \lhd D$ in $M_{23}$ be the inertia and decomposition groups at a closed point $y \in Y_0$.
As $y$ varies, $I$ ranges over the conjugates of $\langle g_1 \rangle$, which are the order~$2$ subgroups of $M_{23}$.
We have $D/I \isom \Gal(f_0) \isom \AGL_3(\F_2)$ by Proposition~\ref{P:irreducible = M23}\eqref{I:fiber above 0},
so $|D| = 2 \lvert \AGL_3(\F_2) \rvert = 2688$.
On the other hand, $D$ is contained in the normalizer $N(I)$, which also has order $2688$, so $D=N(I)$.
The map of transitive $M_{23}$-sets $Y_0 \to \{\textup{order 2 subgroups of $M_{23}$}\}$ given by $y \mapsto I$ is bijective since corresponding one-point stabilizers $D$ and $N(I)$ coincide.
Also, the residue field degree $[\kappa(y):\Q]$ is $(D:I)$.

For a maximal subgroup $H$, let $X_H = Y/H$.
The fiber $(X_H)_0$ above $\qq=0$ is $Y_0/H \isom \{\textup{order 2 subgroups}\}/H$.
The inertia and decomposition groups at $y$ for $Y \to X_H$ are
$I \intersect H$ and $D \intersect H$ \cite[I.\S7, Proposition~22(a)]{Serre1979},
so if $x \in X_H$ is the image of $y$, then $[\kappa(y):\kappa(x)] = (D \intersect H : I \intersect H)$ and
\begin{equation}
\label{E:residue field degree}
    [\kappa(x):\Q] 
    = \frac{[\kappa(y):\Q]}{[\kappa(y):\kappa(x)]} 
    = \frac{(D:I)}{(D \intersect H : I \intersect H)}
    = \frac{(N(I):I)}{(N(I) \intersect H : I \intersect H)}.
\end{equation}
For each $H$, varying $I$ and computing the right side of \eqref{E:residue field degree} shows that $\kappa(x)=\Q$ occurs only for $H=2^4{:}A_7$, and then only for $I$ in a single $H$-orbit, 
hence for only one point of $(X_H)_0$.
\end{proof}

Conjecture~\ref{C:specializations} is equivalent to the statement that the quotients $Y/H$ (for $H$ a maximal subgroup of $M_{23}$) have no rational points other than the points in Propositions \ref{X(Q)} and \ref{P:frombranch}.

\begin{remark}
\label{R:real point}
By the proof of Proposition~\ref{P:irreducible = M23}\eqref{I:complex conjugation}, there is no $\qq \in \PP^1(\R)$ for which the splitting field of $f_\qq$ over $\R$ is contained in the odd-order group $23{:}11$.
In other words, the genus~$8327$ curve $Y/(23{:}11)$ has no $\R$-point.
\end{remark}

\begin{remark}
\label{R:Q23 point}
Let us prove that there is no $\qq \in \PP^1(\Q_{23})$ for which 
the Galois group of $f_\qq$ over $\Q_{23}$ is contained in $M_{11}$;
equivalently, the genus~253 curve $Y/M_{11}$ has no $\Q_{23}$-point.
As displayed in Table~\ref{table:23}, the valuation $v_{23}(\qq)$ determines the $23$-adic Newton polygon of $F(\qq,V)$, which in turn forces the inertia group to have an element of order $7$, $15$, or $23$.
But $M_{11}$ has no element of order $7$, $15$, or $23$.
\end{remark}

\begin{table}[ht]
\centering
\begin{tabular}{c|c|c}
\toprule
$v_{23}(\qq)$ & slopes of Newton polygon & inertia group \\ 
\midrule
$\le -1$ & $\mathbf{-2/5}$ (width 20),\ $\mathbf{-1/3}$ (width 3) & $\Z/15\Z$ \\
$= 0$    & $\mathbf{-10/23}$ (width 23)  & $23$ divides order\\
$= 1$    & $\mathbf{-11/23}$ (width 23)  & $23$ divides order\\
$= 2$    & $\mathbf{-10/23}$ (width 23)  & $23$ divides order\\
$\ge 3$  & $\mathbf{-3/7}$ (width 21), others $\mathbf{\ge 0}$ & $\Z/7\Z$ \\
\bottomrule
\end{tabular}
\bigskip
\caption{$23$-adic Newton polygons of specializations}
\label{table:23}
\end{table}

\begin{remark}
Table~\ref{table:23} also shows that if $\qq$ is the $T$-coordinate of a rational point on $X$,
then $v_{23}(\qq) \ge 3$.
\end{remark}

\medskip

\appendix
\section{Rigorous calculation of function field Galois groups} 
\label{A:rigorous}

\subsection{Insufficiency of Magma's Galois group routines over function fields}

Magma has two routines for computing the Galois group of a polynomial: 
\begin{itemize}
\item \texttt{GaloisGroup} is fast and we have not heard of cases where it gives the wrong answer, but it is not guaranteed to be 100\% rigorous;
\item \texttt{GaloisProof} is slower, but 100\% rigorous.
\end{itemize}
We used \texttt{GaloisProof} to certify the Galois group of each explicit polynomial over $\Q$ in our paper.
In principle, \texttt{GaloisProof} also works over $\Q(t)$ and $L(t)$, but it did not terminate on our degree~$23$ examples over these fields, so instead we reduced the problem to computing a Galois group over $\F_p(t)$ for a prime $p$.
Unfortunately, it appears that \texttt{GaloisProof} has not been implemented over $\F_p(t)$.

Therefore we use a different approach for the certification, due to Elkies \cite[\S3]{Elkies2013}.
We improve the bounds slightly by using Artin $L$-functions instead of full zeta functions as he did.

\subsection{Character sums}

Let $B$ be a \defi{nice} curve over a finite field $\F_q$, that is, a smooth projective geometrically integral scheme of dimension~$1$ over $\F_q$.
Let $\F_q(B)$ be its function field.
Let $f \in \F_q(B)[x]$ be a separable polynomial of degree $n$.
Its splitting field is the function field of a connected Galois \'etale cover $\pi \colon Y \to U$ of some dense affine open subscheme $U \subset B$.
Let $G$ be the Galois group of $\pi$, so $G \le S_n$.
For each closed point $u \in U$, we have the arithmetic Frobenius conjugacy class $\Frob_u$ in $G$.

By the Chebotarev density theorem, $G$ is determined by the distribution of the classes $\Frob_u$.
Effective versions of Chebotarev density theorem with explicit error terms allow one to determine $G$ with only \emph{finitely many} $\Frob_u$.
Although one could use the count of $\Frob_u$ giving a single conjugacy class, one gets better error bounds by taking \emph{combinations} of conjugacy classes weighted according to the values of a character.
Better error bounds permit us to use a smaller $q$, to shorten the running time.
The following summarizes what we need.

\begin{theorem}
\label{T:Weil}
Let $B$ be a nice curve of genus $g$ over a finite field $\F_q$.
Let $U \subset B$ be a dense affine open subscheme.
Let $G$ be a finite group.
Let $\pi \colon Y \to U$ be a connected $G$-torsor.
Let $G_{\geom} \lhd G$ be the geometric monodromy group.
Let $V$ be a $\C$-representation of $G$, with degree $d$, character $\chi$, and Artin conductor $\ff$.
For $b \in B-U$, let $I_b \le G$ be the inertia group, and $V^{I_b}$ the fixed subspace.
\begin{enumerate}[\upshape (a)]
\item
If $\pi$ is tamely ramified at each point of $B-U$, then 
\[
\deg \ff = \sum_{b \in B-U} (\deg b)(d - \dim V^{I_b}).
\]

\item \label{I:Weil RH}
If $V^{G^{\geom}}=0$, then the Artin $L$-function 
\[
    L_V(T) \colonequals \prod_{\textup{closed points $b \in B$}} \det \bigl( 1 - T^{\deg b} \Frob_b | V^{I_b} \bigr)^{-1}
\]
is a polynomial of the form 
\[
     \prod_{i=1}^{d(2g-2) + \deg \ff} (1-\alpha_i T)
\]
for some $\alpha_i \in \C$ satisfying $|\alpha_i|=q^{1/2}$.

\item \label{I:N_chi bound}
If $V^{G^{\geom}}=0$, then the character sum
\[
    N_\chi \colonequals \sum_{u \in U(\F_q)} \chi(\Frob_u)
\]
satisfies
\[
  |N_\chi| \le \bigl(d(2g-2) + \deg \ff \bigr) \; q^{1/2} + \sum_{b \in (B-U)(\F_q)} \dim V^{I_b}.
\]
\end{enumerate}
\end{theorem}

\begin{proof}
\hfill
\begin{enumerate}[\upshape (a)]
\item
By \cite[VI.\S2, Corollary~1$'$]{Serre1979}, for each $b \in B-U$, the multiplicity of $b$ in $\ff$ is 
$d-\dim V^{I_b}$ plus a Swan conductor term that is $0$ in the tame case.
\item
This is Weil's Riemann hypothesis for Artin L-functions \cite[\S V]{Weil1948}.
\item
Equating coefficients of $T$ in \eqref{I:Weil RH} gives 
\[
    \sum_{u \in U(\F_q)} \chi(\Frob_u) + \sum_{b \in (B-U)(\F_q)} \Tr(\Frob_b | V^{I_b}) 
    \quad = \quad 
    - \sum_{i=1}^{d(2g-2) + \deg \ff} \alpha_i.
\]
The first sum is $N_\chi$.
In the second sum, $\Frob_b|V^{I_b}$ has finite order, so $\bigl\lvert \Tr(\Frob_b | V^{I_b}) \bigr\rvert \le \dim V^{I_b}$.
In the third sum, $|\alpha_i| \le q^{1/2}$ by \eqref{I:Weil RH}.
\qedhere
\end{enumerate}
\end{proof}

\begin{remark}
\label{R:nonabelian simple}
In Theorem~\ref{T:Weil}, $G/G^{\geom}$ is the Galois group of an extension of finite fields, hence cyclic.
Thus, if $G$ is a nonabelian simple group, then $G^{\geom}=G$.
\end{remark}

\begin{remark}
Suppose that $G$ is a nonabelian simple group.
Let $\mathbbm{1}$ be the trivial character, so $\langle \chi,\mathbbm{1} \rangle$ is the multiplicity of $\mathbbm{1}$ in $\chi$.
If $\chi$ is a character of $G$ such that $\langle \chi,\mathbbm{1} \rangle = 0$,
then Remark~\ref{R:nonabelian simple} shows that Theorem~\ref{T:Weil}\eqref{I:N_chi bound} applies, so $N_\chi$ is $O(q^{1/2})$.
On the other hand, $N_{\mathbbm{1}}$ is $q + O(q^{1/2})$ by the Hasse--Weil bound.
Taking a combination shows that for any $\chi$, if $q$ is large enough, then $\langle \chi,\mathbbm{1} \rangle$ is the integer closest to $N_\chi/q$.
\end{remark}

\subsection{Application to \texorpdfstring{$M_{23}$}{M23}}

In our application, we choose a prime $p>23$, and $B=\PP^1_{\F_p}$, and $U=\PP^1_{\F_p} - \{0,\sqrt{-23},-\sqrt{-23}\}$.
We have $f \in \F_p(t)[x]$ defining an \'etale cover of $U$ with ramification of type $1^7 2^8$ above $0$,
and type $23$ above $\pm \sqrt{-23}$.
For simplicity, assume that $\sqrt{-23} \in \F_p$.
Let $G$ be the Galois group of $f$.
For the chosen $p$, we will check that $\disc(f) \in \F_p(t)^{\times 2}$, so $G \le A_{23}$.
We wish to prove that the Galois group $G$ of $f$ is $M_{23}$.

\begin{lemma}
Either $G=M_{23}$ or $G=A_{23}$.
\end{lemma}

\begin{proof}
We have $G \le A_{23}$.
The ramification shows that $2 \cdot 23$ divides $|G|$.
The result now follows from the classification of transitive groups on $23$ elements.
\end{proof}

For $r \ge 0$, let $\binom{[23]}{r}$ be the set of $r$-element subsets of $[23] \colonequals \{1,\ldots,23\}$.
Let $V_r$ be the corresponding permutation representation, and let $\chi_r$ be its character.

\begin{lemma}
\label{L:5-transitive}
Let $G$ be $M_{23}$ or $A_{23}$.
\begin{enumerate}[\upshape (i)]
\item We have $\langle \chi_4,\mathbbm{1} \rangle = 1$ for either group.
\item We have $\langle \chi_5,\mathbbm{1} \rangle = \begin{cases} 2, & \textup{if $G=M_{23}$,} \\ 1, &\textup{if $G=A_{23}$.} \end{cases}$
\end{enumerate}
\end{lemma}

\begin{proof}
The integer $\langle \chi_r,\mathbbm{1} \rangle$ is the number of $G$-orbits in $\binom{[23]}{r}$.
\begin{enumerate}[\upshape (i)]
\item
Both $M_{23}$ and $A_{23}$ are $4$-transitive, so the number of orbits in $\binom{[23]}{4}$ is $1$.
\item
The group $A_{23}$ is $5$-transitive, but $M_{23}$ has \emph{two} orbits on $\binom{[23]}{5}$:
some $5$-element subsets are contained in a $7$-element block of $S(4,7,23)$, and some are not.\qedhere
\end{enumerate}
\end{proof}

There is an injection $V_4 \injects V_5$ sending each $4$-element subset to the formal sum of the $5$-element subsets containing it \cite{gottlieb1966}.
Thus $\eta \colonequals \chi_5 - \chi_4$ is the character of a representation.
By Lemma~\ref{L:5-transitive}, the multiplicity $\langle \eta,\mathbbm{1} \rangle$ is $1$ if $G=M_{23}$, and $0$ if $G=A_{23}$.
Define $\chi$ by 
\[
\eta = \begin{cases} \mathbbm{1} + \chi, & \textup{if $G=M_{23}$}, \\ \chi, &\textup{if $G=A_{23}$}. \end{cases}
\]
Thus $\chi$ is a character of a representation $V$ to which Theorem~\ref{T:Weil} applies.

Let us do the computations for $A_{23}$.
First, $\dim V = \binom{23}{5}-\binom{23}{4} = 24794$.
For each branch point $b$, the cycle type of the $I_b$-generator on $[23]$ 
determines the cycle types on $\binom{[23]}{5}$ and $\binom{[23]}{4}$,
from which we calculate $\dim V^{I_b}$.
Now Theorem~\ref{T:Weil} calculates $\deg \ff$ and a bound on $|N_\chi|$.
For $M_{23}$, we have $\chi = \eta - \mathbbm{1}$, so $\dim V$ and each $\dim V^{I_b}$ are $1$ smaller.
See Table~\ref{tab:M23 vs A23}.

\begin{table}[h]
\centering
\begin{tabular}{c|c|ccc|c|c}
\toprule
$G$ & \#orbits on $\binom{[23]}{5}$ & $\dim V$ & $\dim V^{I_0}$ & $\dim V^{I_{\pm\sqrt{-23}}}$ & $\deg \ff$ & bound on $N_\chi$ \\
\midrule
$M_{23}$ & 2 & 24793 & 12529 & 1077 & 59696 & $10110 p^{1/2} + 14683$ \\
$A_{23}$ & 1 & 24794 & 12530 & 1078 & 59696 & $10108 p^{1/2} + 14686$ \\
\bottomrule
\end{tabular}
\bigskip
\caption{Data for $M_{23}$ vs.\ $A_{23}$}
\label{tab:M23 vs A23}
\end{table}

Let $N=N_\eta$.
If $G=M_{23}$, then $N = (p-2) +  N_\chi$.

\begin{corollary}
\label{C:ranges}
\hfill
\begin{itemize}
\item
If $G=M_{23}$, then $|N - (p-2)| \le 10110 p^{1/2} + 14683$.
\item
If $G=A_{23}$, then $|N| \le 10108 p^{1/2} + 14686$.
\end{itemize}
\end{corollary}

If $p$ is large enough, the ranges for $N$ in the two cases of Corollary~\ref{C:ranges} do not overlap, so $N$ determines $G$.
In particular, $p=411000011$ is large enough, and $\sqrt{-23} \in \F_p$.

It remains to compute $N$.
Given $u \in U(\F_p)$, the degrees of the factors of the specialization of $f$ at $u$ tell us the cycle type of $\Frob_u$ (except that an easy local Galois group calculation may be needed if the specialization is not separable).
The cycle type gives $\eta(\Frob_u)$.
We compute $N$ by looping over all $u \in U(\F_p)$.
This takes time, but it requires very little memory and parallelizes easily.
For the reductions of the curves over $\Q(t)$ and $L(t)$ from Sections \ref{S:verification} and \ref{S:Galois on Nielsen}, respectively, we find $N=411207463$ and $N=411071392$, which are both in the $M_{23}$ range.

\section*{Use of AI} 

No text in this article was written by AI.
Our work relies on longstanding domain knowledge and techniques for computing curves of positive genus \cite{Klug-Musty-Schiavone-Voight2014}.  The key strategic decisions and the conceptual analysis for this project were made by the human team. 

AI served as a tool that accelerated a large part of our work, particularly the computations. We used Claude Fable 5 and 5.1, Claude Opus 4.8, and ChatGPT 5.6 Sol for searching the literature, generating code, checking results, and proofreading. The workflow that eventually succeeded was a tightly controlled collaboration in
which the computational work was much faster with AI, while the human collaborators continued to direct the research and verify the results. AI systems increased our confidence in tackling difficult problems and helped us communicate more efficiently across our different areas of expertise.

Our analysis of the literature alongside our new computations suggested that the $4$-branch-point $M_{23}$-covers $X \to \PP^1$ most likely to descend to a low-degree number field were those obtained by pullback along a degree~2 morphism $\PP^1 \to \PP^1$ from a $3$-branch-point cover. We therefore turned our focus to the most promising of the $3$-branch-point covers,  those with local monodromy $(2,\text{23A},\text{23B})$.

The computational work needed to construct the polynomials took 11 CPU-hours on a single core of a 12th Gen Intel Core i5-12600KF processor.
We verified our final results in Magma 2.29-9 and PARI/GP 2.17.4 without the use of AI;
the calculations (done in Magma) in Appendix~\ref{A:rigorous} for each of the two curves took 40 CPU-hours on a single core of an AMD EPYC 7713 processor.

\section*{Acknowledgments} 

We are very grateful to the American Institute of Mathematics for inviting us to participate in the May 27--30, 2026 workshop ``AI and Number Theory,'' organized in collaboration with the Institute for Computer-Aided Reasoning in Mathematics; we began our research at this workshop.
We are also very grateful to workshop participants Michele Tarquini and Zhiyu Zhang for discussions when we were learning the theory and exploring approaches at the beginning of our project.
We thank Anthropic for providing 
workshop participants with a three-month Claude Max subscription. We thank Jean-Pierre Serre for suggestions on the exposition. 
S. Zhang is especially grateful to his Ph.D. advisor, Tony Yue Yu, for generously supporting and encouraging his independent research, and to his mentor, Liam Li, for insightful discussions on AI and encouragement to pursue this problem.

\bibliographystyle{amsalpha}
\bibliography{M23}

\end{document}